\documentclass[11pt]{amsart}
\usepackage{amssymb,amsmath,amsthm,float,enumerate}
\usepackage[margin=1.25in]{geometry}
\usepackage[T1]{fontenc}
\usepackage[utf8]{inputenc}

\newcounter{thmc}
\newcounter{propc}

\newcounter{conjc}
\newtheorem{thm}[thmc]{Theorem}
\newtheorem{lem}[propc]{Lemma}
\newtheorem{prop}[propc]{Proposition}
\newtheorem{corollary}[thmc]{Corollary}
\newtheorem*{cl}{Claim}
\newtheorem{problem}{Problem}
\newtheorem{conj}[conjc]{Conjecture}
\newtheorem{definition}{Definition}

\newcommand{\F}{\mathcal{F}}
\newcommand{\PP}{\mathcal{P}}
\newcommand{\B}{\mathcal{B}}
\newcommand{\D}{\mathcal{D}}
\DeclareMathOperator{\sat}{sat}
\DeclareMathOperator{\isat}{sat^\ast}
\DeclareMathOperator{\bc}{bc}

\usepackage{tikz}
\usetikzlibrary{shapes}

\usetikzlibrary{arrows}
\usetikzlibrary{positioning}
\usetikzlibrary{calc}
\usetikzlibrary{patterns,snakes}

\begin{document}

\title[The Saturation Number of Induced Subposets of the Boolean Lattice]{The Saturation Number of Induced Subposets of the Boolean Lattice}

\author[Ferrara]{Michael Ferrara}
\address{Department of Mathematical and Statistical Sciences, University of Colorado, Denver, CO 80204 (Ferrara, Sullivan)}

\author[Kay]{Bill Kay}
\address{Department of Mathematics and Computer Science, Emory University, Atlanta, GA 30322 (Kay)}

\author[Kramer]{Lucas Kramer$^*$}
\address{Department of Mathematics, Briar Cliff University, Sioux City, IA 51104 (Kramer)}

\author[Martin]{Ryan R. Martin}
\address{Department of Mathematics, Iowa State University, Ames, IA 50011 (Martin)}

\author[Reiniger]{Benjamin Reiniger}
\address{Department of Applied Mathematics, Illinois Institute of Technology, Chicago, IL 60616 (Reiniger)}

\author[Smith]{Heather C. Smith$^{\dagger}$}
\address{School of Mathematics, Georgia Institute of Technology, Atlanta, GA 30332 (Smith)}

\author[Sullivan]{Eric Sullivan}


\begin{abstract}
Given a poset $\PP$, a family $\F$ of elements in the Boolean lattice is said to be $\PP$-saturated if (1) $\F$ contains no copy of $\PP$ as a subposet and (2) every proper superset of $\F$ contains a copy of $\PP$ as a subposet. The maximum size of a $\PP$-saturated family is denoted by ${\rm La}(n,\PP)$, which has been studied for a number of choices of $\PP$.  The minimum size of a $\PP$-saturated family, $\sat(n,\PP)$, was introduced by Gerbner et al. (2013), and parallels the deep literature on the saturation function for graphs.  

   We introduce and study the concept of saturation for induced subposets.  As opposed to induced saturation in graphs, the above definition of saturation for posets extends naturally to the induced setting.  We give several exact results and a number of bounds on the induced saturation number for several small posets.  We also use a transformation to the biclique cover problem to prove a logarithmic lower bound for a rich infinite family of target posets.
\end{abstract}

\let\thefootnote\relax\footnotetext{$^{\dagger}$Corresponding author}
\let\thefootnote\relax\footnotetext{$^{*}$Much of the research done by L. Kramer was completed while affiliated with Bethel College, North Newton, KS and with Iowa State
University, Ames, IA.}
\let\thefootnote\relax\footnotetext{Email addresses: \tt{\{michael.ferrara,eric.2.sullivan\}@ucdenver.edu; bill.w.kay@gmail.com; 
 lucas.kramer@briarcliff.edu; 
 rymartin@iastate.edu; 
 breiniger@iit.edu; 
 heather.smith@math.gatech.edu.}}

\maketitle
\section{Introduction}

A \emph{partially ordered set} (henceforth \emph{poset}) $\mathcal{P}=(P, \leq)$ consists of a set of elements $P$ and a binary relation $\leq$ that is reflexive, transitive, and antisymmetric. We say that a poset $\PP' = ({P}', \leq')$ is a {\em subposet}, sometimes called a {\em weak subposet},  of $\PP = ({P}, \leq)$  if there exists an injective function $f: {P}' \rightarrow {P}$ such that if $u \leq' v$ in $\PP'$ then $f(u) \leq f(v)$ in $\PP$. If a poset $\mathcal{Q}$ does not contain a target poset $\PP$ as a subposet, we say that $\mathcal{Q}$ is $\PP$-free.  The study of $\PP$-free posets dates back to Sperner's Theorem \cite{Sperner}.

The $n$-dimensional Boolean Lattice, $\B_n$, denotes the poset ($2^{[n]}, \subseteq)$ that consists of all subsets of $[n]:= \{1,2,\ldots,n\}$ ordered by inclusion.  
A significant amount of research focuses on ${\rm La}(n, \PP)$, the size of the largest family of elements in $\B_n$ (ordered by inclusion) that is $\PP$-free, and is an analogue to the classical extremal function in graph theory.  We refer the interested reader to  \cite{griggsliposet} for a thorough survey.

In this paper, we are interested in the minimum size of a poset that is maximal with respect to being $\PP$-free. In particular, given a host poset $\mathcal{Q}=(Q,\leq_Q)$, a target poset $\PP=(P, \leq_P)$, and a family $\F \subseteq Q$, we say $\F$ is {\em $\PP$-saturated in $\mathcal{Q}$} if the following two properties hold: 
	\begin{itemize}
		\item $\PP$ is not a subposet of $\F$ (ordered by the restriction of $\leq_Q$ to $\F \times \F$), and 
		\item for any $\F'$ with $\F \subsetneq \F'\subseteq  Q$, $\PP$ is a subposet $\F'$  (ordered by the restriction of $\leq_Q$ to $\F' \times \F'$).
	\end{itemize}  
	For $n \geq 2$ and poset $\PP$, one can describe the extremal function ${\rm La}(n,\PP)$ as the maximum size of a $\PP$-saturated family in $\B_n$. Focusing on the case where the host poset is the Boolean lattice, $\B_n$, we define $\sat(n,\PP)$ to be the \textit{minimum} size of a family that is $\PP$-saturated in $\B_n$.  This quantity, $\sat(n,\PP)$, is termed the \textit{saturation number} or \textit{saturation function} of $\PP$.  

The study of saturation problems in posets parallels the rich literature on graph saturation problems. Given graphs $G$ and $H$, we say that $G$ is {\em $H$-saturated} when $G$ does not contain $H$ as a subgraph, but for any edge $e\not\in E(G)$, the graph $G'=(V(G), E(G) \cup \{e\})$ contains $H$.  The classical extremal number ${\rm ex}(H,n)$ is the maximum number of edges in an $H$-saturated graph with $n$ vertices.  Of interest here is the minimum number of edges in an $H$-saturated graph, denoted $\sat(H,n)$.  
Erd\H{o}s, Hajnal and Moon \cite{EHM} introduced this concept, determined $\sat(K_t,n)$, and characterized the unique saturated graph of minimum size for all $n$ and $t$.  For a thorough survey of saturation in graphs and hypergraphs, we refer the reader to \cite{FFS11}.

Gerbner et al. \cite{GKLPPP13} introduced the saturation function in the setting of posets.  Let $\PP_k$ denote the $k$-element chain, which is the poset with $k$ elements in which each pair of elements is comparable.  A main result of \cite{GKLPPP13} is the following.  
\begin{thm}[Gerbner et al. ~\cite{GKLPPP13}]
	For $n$ sufficiently large, 
	\begin{align*}
		2^{k/2-1}\leq \sat(n,\PP_{k+1}) \leq 2^{k-1}.
	\end{align*}
\end{thm}
Shortly thereafter, the upper bound was improved using an iterative construction.  
\begin{thm}[Morrison, Noel, Scott ~\cite{MNS14}] \label{MNS}
There exists $\varepsilon > 0$ such that for all $k>0$ and for $n$ sufficiently large, 
	\begin{align*}
		\sat(n, \PP_{k+1}) \leq 2^{(1-\varepsilon)k}.
	\end{align*}
\end{thm}

We call $\PP'$ an {\em induced subposet} of $\PP$ if there exists an injective function $f: {P}' \rightarrow {P}$ such that $u \leq' v$  \emph{if} and only if $f(u) \leq f(v)$. For any family $\F\subseteq P$, the {\em poset induced  by $\F$ in $\PP$} is precisely the induced subposet $(\mathcal{F}, \leq')$ of $\PP$ where $\leq'$ is the restriction of $\leq$ to $\F\times \F$.  Because  $\PP_{k+1}$ has no incomparabilities,  $\PP_{k+1}$ is a subposet of poset $\PP$ if and only if it is an induced subposet of $\PP$. However, this is not the case in general.  

In this  paper, we introduce a notion of induced-$\PP$-saturation. In the case of the chain $\PP_k$, our definition aligns with the framework in \cite{GKLPPP13} and \cite{MNS14}. However, as might be expected, we show that this notion is substantially more involved than $\PP$-saturation for many other choices of $\PP$. 
 
\subsection{Induced Subposets}

Before we formally define the notion of induced saturation in posets, we require some terminology from poset theory.  For any poset $\PP=(P,\leq)$  and distinct $x,y\in P$, we use $x \parallel y$ to denote that $x$ is incomparable to $y$ (i.e. $x\not\leq y$ and $y\not\leq x$). We say that $x$ {\em is covered by} $y$ (also $y$ \emph{covers} $x$) if $x\leq y$ ($x\neq y$) and there is no $z\in P\setminus \{x,y\}$ such that $x\leq z\leq y$. 

The {\em Hasse diagram} of $\PP=(P, \leq)$ is a diagram where the elements in $P$ are represented as points in the plane, with $x$ lower than $y$ and a line segment drawn from $x$ upwards to $y$ if and only if $y$ covers $x$. It is frequently convenient to represent $\PP$ by a Hasse diagram.  The posets with Hasse diagrams displayed in Figure~\ref{Hasse_Diagrams} are of particular interest here. 
 
\begin{figure}[H]
\begin{center}
	\begin{tikzpicture}
		\draw [above] (0,0) node { 
			\begin{tikzpicture}
				\tikzset{vertex/.style = {shape=circle,draw,fill,minimum size=.2cm, inner sep = 0 }}
				\node[vertex] (a) at  (0,0) {};
				\path (a) ++(.5,1) node [vertex] (b) {};
				\path (a) ++(-.5,1) node [vertex] (c) {};
				\draw (b)--(a)--(c);
			\end{tikzpicture}
		};
		\draw [below] (0,0) node {$\mathcal{V}_2$};
	
		\draw [above] (3,0) node { 
			\begin{tikzpicture}
				\tikzset{vertex/.style = {shape=circle,draw,fill,minimum size=.2cm, inner sep = 0 }}
				\node[vertex] (a) at  (0,0) {};
				\path (a) ++(.5,-1) node [vertex] (b) {};
				\path (a) ++(-.5,-1) node [vertex] (c) {};
				\draw (b)--(a)--(c);
			\end{tikzpicture}
		};
		\draw [below] (3,0) node {$\Lambda_2$};

		\draw [above] (6,0) node {
			\begin{tikzpicture}[x=.7cm,y=.7cm]
				\tikzset{vertex/.style = {shape=circle,draw,fill,minimum size=.2cm, inner sep = 0 }}
				\node[vertex] (a) at  (0,0) {};
				\path (a) ++(-.5,.75)  node [vertex] (b) {};
				\path (a) ++(.5,.75)  node [vertex] (c) {};
				\path (a) ++(0,1.5)  node [vertex] (d) {};
				\draw (a)--(b)--(d)--(c)--(a);
			\end{tikzpicture}
		};
		\draw [below,align=center] (6,0) node {$\D_2$,\\ the diamond};

		\draw [above] (9,0) node {
			\begin{tikzpicture}
				\tikzset{vertex/.style = {shape=circle,draw,fill,minimum size=.2cm, inner sep = 0 }}
				\node[vertex] (a) at  (0,0) {};
				\node[vertex] (b) at  (0,1) {};
				\node[vertex] (c) at  (1,0) {};
				\node[vertex](d) at (1,1){};
				\draw (a)--(b)--(c)--(d);
				\node at (-0.3,-0.1) {$A$};
				\node at (-0.3,0.8) {$B$};
				\node at (1.3,-0.1){$C$};
				\node at (1.3,0.8) {$D$};
			\end{tikzpicture}
		};
		\draw [below] (9,0) node {$\mathcal{N}$};

		\draw [above] (12,0) node {
			\begin{tikzpicture}
				\tikzset{vertex/.style = {shape=circle,draw,fill,minimum size=.2cm, inner sep = 0 }}
				\node[vertex] (a) at  (0,0) {};
				\node[vertex] (b) at (0,1) {};
				\node[vertex] (c) at (1,1) {};
				\node[vertex] (d) at (1,0) {};
				\draw (a) --(b)--(d) -- (c) --(a);
				\node at (-0.3,-0.1) {$A$};
				\node at (-0.3,0.8) {$B$};
				\node at (1.3,-0.1){$C$};
				\node at (1.3,0.8) {$D$};
			\end{tikzpicture}
		};
		\draw [below,align=center] (12,0) node {$\bowtie$,\\ the butterfly};
	\end{tikzpicture}
	\end{center}
	\caption{The Hasse diagrams of several posets of interest}\label{Hasse_Diagrams}
\end{figure}
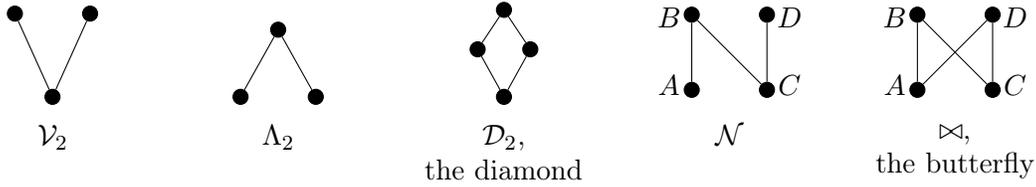

\subsection{Induced Poset Saturation} The focus of this paper is to explore the notion of poset saturation for induced subposets.  Given a host poset $\mathcal{Q}=(Q,\leq_Q)$, a target poset $\PP=(P, \leq_P
)$, and a family $\F \subseteq Q$, we say $\F$ is {\em induced-$\PP$-saturated in $\mathcal{Q}$} if the following two properties hold: 
	\begin{itemize}
		\item the poset induced by $\F$ in $\mathcal{Q}$ does not contain an induced copy of $\PP$, and 
		\item for every $\F'$ with $\F \subsetneq \F' \subseteq \mathcal{Q}$, the poset induced by $\F'$ in $\mathcal{Q}$ contains an induced copy of $\PP$.
	\end{itemize}
Given $n \geq 2$ and a poset $\PP = (P, \leq)$, let $\isat(n,\PP)$ be the minimum size of a family $\F$ that is induced-$\PP$-saturated in $\B_n$. If $\PP$ is not an induced subposet of $\B_n$, then $\isat(n,\PP)=2^n$, vacuously. Note that $\isat(n,\PP)<2^n$ when $n>|P|$, the proof of which is left to the reader.

It is worthwhile to note that induced saturation is less straightforward to define in graphs than in posets. For instance, determining a notion of induced saturation that encompasses $\overline{K}_t$, a set of $t$ pairwise nonadjacent vertices, as a potential target graph is nontrivial.  Indeed, in the natural first extension of the definition of an $H$-saturated graph to the induced setting, $\overline{K}_t$ has no ``induced" saturation number when the number of vertices in the host graph is at least $t$.  Hence induced saturation for graphs needs additional ideas, such as in \cite{MS}, where the authors utilize trigraphs \cite{ChudTriG} to address such issues.  As we discuss below, however, the natural extension of the saturation function for posets to the induced framework easily extends to antichains, which are the nearest poset analogue to independent sets in graphs.  

As is often the case with graph saturation problems, induced saturation in posets is not monotone. In particular, one cannot conclude that if $\PP$ is an induced subposet of $\mathcal{Q}$, then $\isat(n,\PP) \leq \isat(n, \mathcal{Q})$. For example, let $\PP$ be the poset with two elements, $a$ and $b$, which are incomparable. It fairly easy to see that $\isat(n,\PP) = n+1$. On the other hand, let $\mathcal{Q}$ be the poset with 3 elements, $a$, $b$, and $c$, such that $a\leq c$, but $b$ is incomparable to both $a$ and $c$. The family $\F=\{\emptyset, \{1\}, [n]\setminus \{1\}, [n]\}$ is induced-$\mathcal{Q}$-saturated in $\B_n$.

\subsection{Summary of Results}  Next, we summarize the results that will be proven in Sections~\ref{sec:upper} and~\ref{LBs}.  To begin, let $\mathcal{A}_k$ denote the $k$-antichain, which is the poset with $k$ elements in which each pair of elements is incomparable.  Note that  $\sat(n,\mathcal{A}_{k+1})= k$ since any saturated family that has no (weak) copy of $\mathcal{A}_{k+1}$ has at least $k$ elements and any family that has at least $k+1$ elements contains a (weak) copy of $\mathcal{A}_{k+1}$. 
In contrast,  the induced saturation number of the $k$-antichain is linear in $n$:

\begin{thm}\label{thm:antichain}
	If $n>k\geq 3$, then
	\begin{align*}
		3n-1 \leq \isat(n,\mathcal{A}_{k+1}) \leq (n-1)k-\left( \frac{1}{2} \log_2 k + \frac{1}{2} \log_2 \log_2 k + O(1)\right). 
	\end{align*}
    In particular, $\isat(n, \mathcal{A}_{k+1})=\Theta(n)$.
\end{thm}

Next, note that $\sat(n, \mathcal{V}_2) = 2$ because any family that has no (weak) $\mathcal{V}_2$ has at least $2$ elements and the family $\F =\{\emptyset, \{1\}\}$ is $\mathcal{V}_2$-saturated. 
By symmetry, $\sat(n, \Lambda_2) = 2$.  
But  again the induced saturation number is linear in $n$: 

\begin{thm}\label{thm:veewedge}
	If $n\geq 2$, then
	\begin{align*}
		\isat(n,\mathcal{V}_2) = \isat(n,\Lambda_2) = n+1.
	\end{align*}
\end{thm}  

The diamond ($\D_2$ in Figure~\ref{Hasse_Diagrams}) further emphasizes the distinction between saturation numbers and induced saturation numbers. 
We note that $\sat(n, \D_2) = 3$ since any $\D_2$-saturated family has at least $3$ elements and the family  $\F= \{\emptyset,\{1\},[n]\}$ is $\D_2$-saturated. 
However,  in the induced setting, we have Theorem~\ref{thm:diamond}:

\begin{thm}\label{thm:diamond}
	If $n \geq 2$, then
	\begin{align*}
		\lceil \log_2 n \rceil \leq \isat(n,\D_2)\leq n+1.
	\end{align*}
\end{thm}

We note that $\sat(n,\mathcal{N})=3$, as any $\mathcal{N}$-saturated family has at least $3$ elements, and the family  $\F=\{\emptyset,\{1\},[n]\}$ is $\mathcal{N}$-saturated. 
On the other hand, we prove Theorem~\ref{thm:theN}:

\begin{thm}\label{thm:theN}
	If $n \geq 3$, then
	\begin{align*}
		\lceil \log_2 n \rceil \leq \isat(n,\mathcal{N})\leq 2n.
	\end{align*}
\end{thm}

The last particular poset we would like to highlight is the butterfly ($\bowtie$ in Figure \ref{Hasse_Diagrams}). 
We note that $\sat(n,\bowtie) = 4$ as any $\bowtie$-saturated family has at least $3$ elements and a little more work gives a lower bound of 4. For the upper bound, the family  $\F = \{\emptyset, \{1\}, [n]\setminus \{1\}, [n]\}$ is $\bowtie$-saturated. 
However, Theorem~\ref{thm:butterfly} indicates different behavior for the induced saturation number:

\begin{thm} \label{thm:butterfly} If $n\geq 3$, then 
\begin{align*}
		\lceil \log_2 n \rceil \leq \isat(n, \bowtie) \leq \binom{n}{2} +2n - 1.
	\end{align*}
\end{thm}

In each of the above instances, we remark that $\sat(n,\PP)$ is bounded by a constant, while $\liminf_{n\rightarrow\infty} \isat(n, \PP)=\infty$. 
However, unlike $\mathcal{A}_{k+1}$, $\mathcal{V}_2$,  and $\Lambda_2$, we do not know the growth rate of  $\isat(n,\PP)$    for the diamond, $\mathcal{N}$, or the butterfly. 

Next we define an infinite class of posets ${\mathbb P}$ and show in Theorem \ref{thm:phat} that the induced saturation number of each $\PP\in{\mathbb P}$ is at least logarithmic in $n$.   This class includes the diamond, $\mathcal{N}$, and the butterfly, so that the lower bounds in Theorems \ref{thm:diamond} - \ref{thm:butterfly} all follow from Theorem \ref{thm:phat}. 

\begin{definition}[Unique Cover Twin Property]\label{E-UCTP}\label{phat}\label{UCTP}
	Let $\PP = (P,\leq)$ be a poset. We say that $\PP$ has the {\em Unique Cover Twin Property} (UCTP)  if for every element $S \in P$ that has precisely one cover $T \in P$ there exists $S' \in  P$, $S'\neq S$, such that $T$ also covers $S'$. Let ${\mathbb{P}}$ denote the collection of all posets that have at least 2 elements and have UCTP. 
\end{definition}

Notice that all posets in Figure~\ref{Hasse_Diagrams} have UCTP.
Our main result, Theorem~\ref{thm:phat}, provides a logarithmic lower bound on the saturation numbers of all non-trivial posets with UCTP.

\begin{thm} \label{thm:phat}
	If $\PP\in {\mathbb P}$, then 
	\begin{align*}
		\lceil \log_2 n \rceil \leq \isat(n,\PP).
	\end{align*}
\end{thm}

The proof of this result appears in Section~\ref{LBs}, and follows from an
 analysis of the biclique cover number of an auxiliary graph that will be defined later.  

\section{Upper Bounds On Induced Saturation Numbers}
\label{sec:upper}

In order to prove an upper bound on $\isat(n,\PP)$ for some poset $\PP$ and some integer $n\geq 1$, we find a family $\F$ in $\B_n$ that is induced-$\PP$-saturated.  For instance, consider the family $\F$ composed of $k$ maximum chains in $\mathcal{B}_n$ ($n\geq k$) whose pairwise intersection is precisely  $\{\emptyset,[n]\}$. This family $\F$ has cardinality $k(n-1)+2$ and is easily seen to be induced-$\PP$-saturated for $\PP\in\{\mathcal{A}_{k+1}, \mathcal{V}_{k+1}, \Lambda_{k+1}, \D_{k+1}\}$ (see Figure~\ref{fig:fullchains}) implying Proposition \ref{fullchains}.

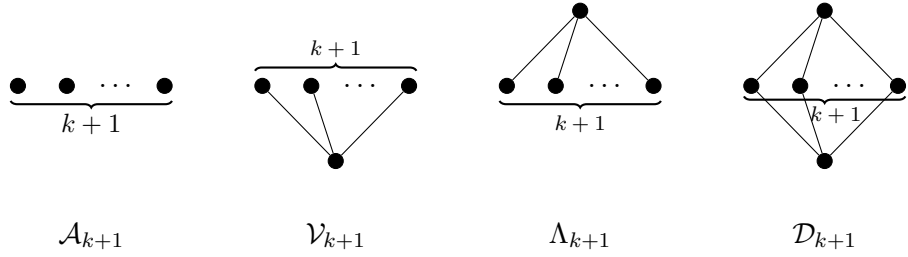
\begin{figure}[h]
	\begin{center}
	\begin{tikzpicture}[x=.65cm]
		\tikzset{vertex/.style = {shape=circle,draw,fill,minimum size=.2cm, inner sep = 0 }}
		\node[vertex] (ba) at (0,0) {};
		\node[vertex] (bb) at (1,0) {};
		\node (bc) at (2,0) {$\ldots$};
		\node[vertex] (bd) at (3,0) {};
		\node (k) at (1.5, -.5) {\small{$k+1$}};
		\node (A) at (1.5, -2) {$\mathcal{A}_{k+1}$};
		\draw[thick, decorate, decoration = {brace, mirror, raise = .2cm}] (-.15,0) -- (3.15,0);

		\node[vertex] (ba) at (5,0) {};
		\node[vertex] (bb) at (6,0) {};
		\node (bc) at (7,0) {$\ldots$};
		\node[vertex] (bd) at (8,0) {};
		\node[vertex] (bottom) at (6.5, -1) {};
		\node (k) at (6.5, .5) {\scriptsize{$k+1$}};
		\node (V) at (6.5, -2) {$\mathcal{V}_{k+1}$};
		\draw (ba) -- (bottom);
		\draw (bd) -- (bottom);
		\draw (bb) -- (bottom);
		\draw[thick, decorate, decoration = {brace, raise = .2cm}] (4.85,0) -- (8.15,0);

		\node[vertex] (a) at (10,0) {};
		\node[vertex] (b) at (11,0) {};
		\node (c) at (12,0) {$\ldots$};
		\node[vertex] (d) at (13,0) {};
		\node[vertex] (top) at (11.5, 1) {};
		\node (k) at (11.5, -.5) {\scriptsize{$k+1$}};
		\node (L) at (11.5, -2) {$\Lambda_{k+1}$};
		\draw (top) -- (a);
		\draw (top) -- (b);
		\draw (top) -- (d);
		\draw[thick, decorate, decoration = {brace, mirror, raise = .2cm}] (9.85,0) -- (13.15,0);

		\node[vertex] (ba) at (15,0) {};
		\node[vertex] (bb) at (16,0) {};
		\node (bc) at (17,0) {$\ldots$};
		\node[vertex] (bd) at (18,0) {};
		\node[vertex] (bottom) at (16.5, -1) {};
		\node (k) at (16.727, -.38) {\scriptsize{$k+1$}};
		\node (D) at (16.5, -2) {$\D_{k+1}$};
		\draw (ba) -- (bottom);
		\draw (bd) -- (bottom);
		\draw (bb) -- (bottom);
		\node[vertex] (top) at (16.5,1) {};
		\draw (ba) -- (top);
		\draw (bd) -- (top);
		\draw (bb) -- (top);
		\draw[thick, decorate, decoration = {brace,  mirror,raise = .14cm}] (14.85,0) -- (18.15,0);
	\end{tikzpicture}
	\end{center}
	\caption{The Hasse diagrams of posets whose induced saturation numbers are upper bounded in  Proposition~\ref{fullchains}.}
	\label{fig:fullchains}
\end{figure}

\begin{prop}\label{fullchains}
	If $n > k$ and $\PP\in\{\mathcal{A}_{k+1},\mathcal{V}_{k+1},\Lambda_{k+1},\D_{k+1}\}$, then 
	\begin{align*}
		\isat (n, \PP)\leq k(n-1)+2.
	\end{align*}
\end{prop}

Proposition \ref{fullchains} provides the upper bounds  in Theorems  \ref{thm:veewedge} and \ref{thm:diamond}.   
Note that when $k=1$, the family $\F'=\{\{i\}: i\in [n]\} \cup \{\emptyset\}$ is also induced-$\D_2$-saturated. 
Hence, extremal examples need not be unique. 

While Proposition~\ref{fullchains} provides an upper bound for $\isat(n,\mathcal{A}_{k+1})$, this is not best possible. For $n>k$, suppose $k=\binom{\ell}{(\ell-1)/2}$ for some odd integer $\ell$. (Note: With care, this construction can be tweaked to remove this restriction on $k$.)

Consider the following families of elements of $\mathcal{B}_n$: 
\begin{align*}
\mathcal{J}&:=\left\{S: S \subseteq \{1,2,\ldots, \ell\}, |S| \leq \frac{\ell-1}{2}\right\},\\
\mathcal{K}&:= \left\{S \cup \{\ell+1, \ell+2, \ldots, \ell+t\}: t\in \{1,2, \ldots, n-\ell\}, S \subseteq \{1,2,\ldots, \ell\} , |S|=\frac{\ell-1}{2}\right\},\\
\mathcal{L}&:= \left\{S \cup \{\ell+1, \ell+2, \ldots, n\}: S\subseteq \{1,2,\ldots, \ell\}, |S| \geq \frac{\ell+1}{2}\right\}.
\end{align*}

One can view this as the lower half of the Boolean lattice $\mathcal{B}_\ell$ (described by $\mathcal{J}$) and a copy of the top half of $\mathcal{B}_\ell$ (described by $\mathcal{L}$)  connected by $k$ disjoint chains (described by $\mathcal{K}$). 

We claim that $\mathcal{F}=\mathcal{J} \cup \mathcal{K} \cup \mathcal{L}$ is an induced-$\mathcal{A}_{k+1}$-saturated family. To see this, first observe that $\mathcal{F}$ has width $k$ as it can be decomposed into $k$ chains. Thus $\mathcal{F}$ does not contain a copy of $\mathcal{A}_{k+1}$. Now consider any set $T \subseteq [n]$, $T\not \in \mathcal{F}$. We will show that the family $\F \cup \{T\}$ contains an induced copy of $\mathcal{A}_{k+1}$.

If $|T|\in \{\frac{\ell+1}{2}, \ldots, n-\frac{\ell+1}{2}\}$, then $T$ forms an antichain of size $k+1$ with the $k$ elements in $\mathcal{K}$ that have the same size as $T$. 
If $|T|\leq \frac{\ell-1}{2}$, then there exists $j\in T$, $j \not \in \{1,2, \ldots, \ell\}$. So $T$ forms an antichain of size $k+1$ with $\{S: S \subseteq \{1,2,\ldots, \ell\}, |S|=\frac{\ell-1}{2}\}$ from $\mathcal{J}$. 
On the other hand, if $|T|\geq n-\frac{\ell-1}{2}$, then $T$ forms an antichain of size $k+1$ with the sets in $\mathcal{L}$ of size $n-\frac{\ell-1}{2}$.  

The size of this induced-$\mathcal{A}_{k+1}$-saturated family is $2^\ell + (n-\ell)k$. 
From the fact that $k=\binom{\ell}{(\ell-1)/2}$, the value of $\ell$ is asymptotically $\log_2 k + \frac{1}{2}\log_2 \log_2 k +O(1)$ which implies $\isat(n, \mathcal{A}_{k+1}) \leq (n-1)k-k\left( \frac{1}{2} \log_2 k + \frac{1}{2} \log_2 \log_2 k + O(1)\right)$. This is an improvement of the upper bound of $(n-1)k+2$ obtained from the $k$ disjoint chains construction. However they are asymptotically the same in $n$. 

Next we turn our attention to the $\mathcal{N}$ poset. For consistency, let $\mathcal{N}(A,B,C,D)$ denote the poset $\mathcal{N}$ in Figure~\ref{Hasse_Diagrams}, with $A<B$, $C<B$, and $C<D$. 
Proposition~\ref{UB:TheN} provides the upper bound  in Theorem~\ref{thm:theN}. 

\begin{prop}\label{UB:TheN}
	For $n\geq 3$, $\isat(n,\mathcal{N}) \leq 2n$. 
\end{prop}

\begin{proof}
	Consider the family 
	\begin{align*}
		\F=\{\{i\}: i\in [n]\} \cup \{\{1,2,\ldots,i\}: i\in [n]\} \cup \{\emptyset\}.
	\end{align*}

	Let ${\mathcal N}(A,B,C,D)$ be an induced copy of $\mathcal{N}$ in $\mathcal{B}_n$. We will show that $A, B, C,$ and $D$ cannot all be in $\F$. Because $A$ and $C$ are incomparable, both must have size at least 1. 
	Thus $B$ and $D$, which are incomparable, each have size at least 2. 
	But all sets in $\F$ with size at least 2 lie on a chain and thus are comparable. 
	Therefore $\F$ has no induced copy of $\mathcal{N}$. 

	Further, consider any set $S=\{i_1, i_2, \ldots, i_k\}$, with $i_1 < i_2 < \ldots < i_k$, that is not in $\F$. 
	Note that $k\geq 2$. 
	Choose  the smallest $t\in  [n]\setminus S$. 
	Since $S \not\in \F$, $t< i_k$. 
	If $t \neq 1$, then $\mathcal{N}(\{t\}, \{1,2,\ldots, t\}, \{1\},S)$ is an induced copy of $\mathcal{N}$ in $\F \cup \{S\}$.  
	If $t=1$, then $\mathcal{N}(\{1\}, \{1,2,\ldots,i_1\}, \{i_1\},S)$ is an induced copy of $\mathcal{N}$. 
\end{proof}

We now establish the upper bound for the induced saturation number of the butterfly as stated in Theorem~\ref{thm:butterfly}. 
Let $\bowtie(A,B,C,D)$ denote the butterfly with elements $A$, $B$, $C$, and $D$ where $A$ and  $C$ are the minimal elements. 

\begin{prop}\label{UB:Butterfly}
	For $n\geq 3$, $\isat(n,\bowtie) \leq \binom{n}{2} + 2n - 1$.
\end{prop}

\begin{proof}
	Consider the family $\F= \{\emptyset\} \cup \F_1 \cup \F_2 \cup \F_c$ where 
$\F_1 = \{\{i\}: i\in [n]\}$, $\F_2 = \{ \{i,j\}: i, j\in [n], i\neq j\}$, and $\F_c = \{\{1,2,\ldots, i\}:  3\leq i\leq n\}$. 
	Suppose $\bowtie(A,B,C,D)$ is an induced butterfly in $\F$. 
	 Since the maximal elements must be incomparable and $\F_c$ is a chain, either $B$ or $D$ must be in $\F_2$ and thus the two incomparable minimal elements, $A$ and $C$, are from  $\F_1$.
	However the only supersets of both $A$ and $C$ in $\F$ are $A\cup C$ and sets in $\F_c$ that contain $A\cup C$. But these are all comparable, so $\F$ contains no induced $\bowtie$.   

	 Let $S=\{i_1, i_2, \ldots, i_k\}$, where $i_1< i_2 < \cdots < i_k$, be an element of $\B_n\setminus \F$ (so $k\geq 3$). 
	Further, let $t\in [n]$ be the smallest value not in $S$ and note that $t<i_k$.  
	For $M=\max\{i_2, t\}$, the butterfly $\bowtie(\{i_1\}, \{1,2,\ldots, M\}, \{i_2\}, S)$ is induced.  
	Thus, $\F$ is induced-$\bowtie$-saturated, proving the proposition.
\end{proof}

\section{Lower Bounds on the Induced Saturation Number}\label{LBs}

In this section, we establish lower bounds on the induced saturation number $\isat(n,\PP)$ for specific choices of $\PP$. 
For the poset $\mathcal{A}_{k}$, the lower bound matches the upper bound asymptotically (up to a constant multiple). For  $\mathcal{V}_2$ and $\Lambda_2$, the lower bounds match the upper bounds exactly, and thus we know the induced saturation numbers. 
We conclude with a general logarithmic lower bound on  $\isat(n,\PP)$ for a rich class of posets.

\subsection{Asymptotic and  Exact Induced Saturation Numbers}

One class of posets for which our upper and lower bounds for $\isat(n, \PP)$ match asymptotically is antichains. 
Recall that $\mathcal{A}_k$ denotes the {\em antichain} with $k$ elements, a collection of $k$ pairwise incomparable elements. 
To address $\isat(n, \mathcal{A}_k)$, we will make use of the classical theorem of Dilworth.

\begin{thm}[Dilworth~\cite{Dilworth}] \label{thm:Dilworth}
	Let $\PP= (P, \leq)$ be a (finite) poset. 
	If $k$ is the maximum size of an antichain in $\PP$, then there exist disjoint  chains $\mathcal{C}_1, \mathcal{C}_2, \ldots, \mathcal{C}_k$ such that $P = \mathcal{C}_1 \cup \mathcal{C}_2 \cup \ldots \cup \mathcal{C}_k$. 
\end{thm}

We can now establish the lower bound in Theorem~\ref{thm:antichain}.

\begin{prop}\label{LB:antichain}
	If  $n>k\geq 3$, then
	\begin{align*}
		3n-1 \leq \isat(n,\mathcal{A}_{k+1}) .
	\end{align*}
\end{prop}

\begin{proof}
For contradiction, suppose that $\F$ is an induced-$\mathcal{A}_{k+1}$-saturated family with at most $3n-2$ elements.
Since $k$ is the maximum size of an antichain in $\F$, there exist disjoint  chains $ \hat{\mathcal{C}}_1, \hat{\mathcal{C}}_2, \ldots, \hat{\mathcal{C}}_k$ (not necessarily maximal chains) such that $\F = \hat{\mathcal{C}}_1 \cup \hat{\mathcal{C}}_2 \cup \ldots \cup \hat{\mathcal{C}}_k$ by Dilworth's Theorem (Theorem~\ref{thm:Dilworth}). Let $\mathcal{C}_i$ be an arbitrary extension of $\hat{\mathcal{C}}_i$ to a maximal chain in $\B_n$ for $1 \leq i \leq k$. The maximal chains ${\mathcal{C}}_1, {\mathcal{C}}_2, \ldots, {\mathcal{C}}_k$ form a chain cover of $\F$ (not necessarily disjoint) and remain $\mathcal{A}_{k+1}$-free by Dilworth's Theorem.  Since $\mathcal{F}$ is maximal with respect to containing no induced copy of $\mathcal{A}_{k+1}$, the union of these full chains is $\F$. 

Since $\emptyset$ and $[n]$ are comparable to every element of $\mathcal{B}_n$ and thus cannot participate in any $ (k+1)$-antichain, both are in $\F$.  Hence, $|\F\setminus \{\emptyset, [n]\}| \leq 3n-4$. 

Suppose $\F$ contains only one set, $A$, of cardinality $j$ for some $j\in \{1,2,\ldots, n-1\}$. Then all chains, and in particular $\mathcal{C}_1$ and $\mathcal{C}_2$ contain $A$. Further, $\mathcal{C}_1$ contains $A\setminus \{x\}$, $A$, and $A\cup \{y\}$ for some $x\in A$ and $y\not\in A$. Consider the set $A'=A\cup \{y\} \setminus \{x\}$ of size $j$. In particular, $A' \not\in \F$ and yet if we replace $\mathcal{C}_1$ with $\mathcal{C}'_1:=\mathcal{C}_1 \cup \{A'\} \setminus \{A\}$, we have a cover of $\F \cup \{A'\}$ with $k$ chains. Thus $\F \cup \{A'\}$ does not contain an induced copy of $A_{k+1}$, contradicting the choice of $\F$. So we may assume that $\F$ contains at least 2 sets of each cardinality $j\in \{1,2,\ldots, n-1\}$.

Since $|\F\setminus \{\emptyset, [n]\}| \leq 3n-4$, there must be some $j\in \{1,2,\ldots, n-1\}$ for which $\F$ contains exactly two sets of cardinality $j$, say $A$ and $B$.
Now consider the sets in $\F$ of size $j-1$. If $j>1$, there must be at least two sets $S$ and $T$ in $\F$ that have cardinality $j-1$. (Note: If $j=1$, then consider sets of cardinality $j+1$ rather than $j-1$. A similar argument will hold.) 
 It is not possible to have $S \subseteq A \cap B$ and $T \subseteq A \cap B$ since they are incomparable and both have size $j-1$. Without loss of generality, $S\not\subseteq B$. 

Since $S\in \F$ and $S \not\subseteq B$, there must be a chain, say $\mathcal{C}_1$, that contains $S$ and that chain must also contain $A$.  In particular, for some $x\in A$ and $y\not\in A$, $S=A \setminus \{x\}$, $A$, and $A \cup \{y\}$ are on $\mathcal{C}_1$. Let $A' = A \cup \{y\} \setminus \{x\}$. If there are at least two chains that contain $A$, then replacing $\mathcal{C}_1$ with $\mathcal{C}'_1 = \mathcal{C}_1 \cup \{A'\} \setminus \{A\}$ yields a chain cover of $\F \cup \{A'\}$ with $k$ chains, a contradiction to the choice of $\F$ as before.  
If instead the only chain that contains $A$ is $\mathcal{C}_1$, then at least two chains contain $B$ since $k\geq 3$. In particular, these chains pass through both $T$ and $B$ as $S \not\subseteq B$. Let $\mathcal{C}_2$  be one of these chains. In particular, $\mathcal{C}_2$ contains $T=B \setminus \{z\}$, $B$, and $B\cup \{w\}$. Let $B'=B \cup \{w\} \setminus \{y\}$. If $B' \neq A$, then replace $\mathcal{C}_2$ with $\mathcal{C}'_2 = \mathcal{C}_2 \cup \{B'\} \setminus \{B\}$ to obtain a cover of $\F \cup \{B'\}$ with $k$ chains, a contradiction as before. If $B'=A$, then replace $\mathcal{C}_1$ with $\mathcal{C}'_1$ and replace $\mathcal{C}_2$ with $\mathcal{C}'_2$ to obtain a chain cover of $\F \cup \{A'\}$ with $k$ chains, a contradiction. 
\end{proof}

Proposition~\ref{LB:antichain}  and Proposition~\ref{fullchains}  yield the bounds on $\isat(n,\mathcal{A}_{k+1})$ in Theorem~\ref{thm:antichain}.  It is worth mentioning that $\isat(n,\mathcal{A}_2)=n+1$ and $\isat(n,\mathcal{A}_3)=2n$. The first is clear. For the second, an induced-$\mathcal{A}_3$-saturated family $\F$ can be decomposed into two chains. The proof for Proposition~\ref{LB:antichain} showing that $\F$ contains at least two sets of cardinality $j$ for each $2\leq j \leq n-1$ also holds in this case, giving the desired result.

Now we turn our attention to the diamond, $\D_2$, and begin with the following lemma.  In addition to providing a starting point for proving the lower bound on $\isat(n,\D_2)$, this result also allows us to determine several other induced saturation numbers exactly.  

	For ease of notation, we will use $\D_2(A,B,C,D)$ to denote the diamond on four distinct elements, $A$, $B$, $C$, and $D$  where  $A$ is the unique minimal element, $D$ is the unique maximal element, and $B$ and $C$ are incomparable. 
	Note that $\D_2(A,B,C,D)$ and $\D_2(A,C,B,D)$ describe the same diamond. 	
	
	\begin{lem}\label{diamondempty}
		Let $\F$ be an induced-$\D_2$-saturated family in the Boolean lattice $\B_n$, $n \geq 2$. If $\emptyset\in\F$  or $[n]\in\F$, then $|\F| \geq n+1$. 
	\end{lem}

	\begin{proof}
		 By duality, it suffices to prove the case when $\emptyset\in\F$.
		Fix $\F$, an induced-$\D_2$-saturated family in the Boolean lattice $\B_n$, $n \geq 2$, with $\emptyset \in \F$.		We will define an injection $\phi:[n] \rightarrow  \F\setminus\{\emptyset\}$.  If $\{x'\}\in \F$, let $\phi(x')=\{x'\}$. Consider $x \in [n]$ such that $\{x\}\notin\F$. 
        Because $\F$ is induced-$\D_2$-saturated and $\emptyset \in \F$, there exists $B,T \in \F$ such that $\D_2(\emptyset, \{x\}, B,T)$ is in $\F\cup\{x\}$. In particular, $x\in T\setminus B$. Among all such pairs $(B,T)$, 
        choose the pair with $B$ maximum cardinality, breaking ties by minimizing the cardinality of $T$ and set $\phi(x)=T$. 
        
\begin{cl} $T=B\cup \{x\}$. \end{cl}
\begin{proof}
		For contradiction, suppose $T \neq B \cup \{x\}$. So $B \cup \{x\} \not \in \F$, which implies there exist sets $B'$, $S$, $T'$ $\in \F$ that together with $B \cup \{x\}$ induce a $\D_2$ in $\B_n$. 
		Note that $B\cup \{x\}$ cannot be the minimal element in the diamond formed because then $\D_2(B, B', S, T')$ would also be a diamond formed entirely by elements in $\F$. 
		Similarly, $B\cup \{x\}$ cannot be the maximal element of the diamond because $\D_2(B', S, T', T)$ would be a diamond in $\F$. 
		Thus we have diamond $\D_2(B', S, B \cup \{x\}, T')$ in $\F\cup \{B\cup\{x\}\}$. If more than one such induced diamond is present in $\F\cup \{B\cup\{x\}\}$, consider the one in which $B'$ has maximum cardinality. At present, the Hasse diagram in Figure~\ref{fig:D2helper} shows the general setting with the possibility of more comparabilities. 

		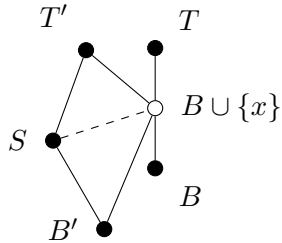
\begin{figure}[H]
			\begin{center}
			\begin{tikzpicture}[x=.8cm,y=.8cm]
				\tikzset{vertex/.style = {shape=circle,draw,fill,minimum size=.2cm, inner sep = 0 }}
				\tikzset{vertex2/.style = {shape=circle,draw,minimum size=.2cm, inner sep = 0 }}
				\node[vertex2, label = {[label distance = .1cm]0:{$B \cup \{x\}$}}] (BuX) at  (0,0) {};
				\node[vertex,  label = {[label distance = .1cm]30:{$T$}}] (T) at (0,1) {};
				\node[vertex,  label = {[label distance = .1cm]-30:{$B$}}] (B) at (0,-1) {};
				\path (BuX) ++(140:1.5) node[vertex,label = {[label distance = .1cm]120:{$T'$}}  ] (Tprime) {};
				\path (Tprime) ++(250:1.6) node[vertex, label = {[label distance = .1cm]180:{$S$}}  ] (S) {};
				\path (S) ++(300:1.7) node[vertex, label = {[label distance = .1cm]180:{$B'$}}  ] (Bprime) {};
				\draw (B) -- (BuX) --(T);
				\draw (BuX) -- (Tprime) -- (S) -- (Bprime)--(BuX);
				\draw [dashed] (S)--(BuX);
			\end{tikzpicture}
			\end{center}
			\caption{Comparabilities (solid lines) and incomparabilities (dashed line) established in the proof of Lemma~\ref{diamondempty}.}
			\label{fig:D2helper}
		\end{figure}

Next, we study the relationship between $B$ and $B'$. 
Since $B' \subsetneq B \cup \{x\}$, either $B' \subseteq B$ or $B'$ is incomparable to $B$. 
However, if $B$ is incomparable to $B'$, then we have the diamond $\D_2(\emptyset, B, B', T')$ in $\F$, a contradiction to the choice of $\F$. 

Next, suppose that $B' \subsetneq B$. Because $S$ and $B\cup \{x\}$ are incomparable, we may conclude $S \not \subseteq B$. 
		Thus, either $B \parallel S$ or $B \subsetneq S$. 
		In the first case, we have the diamond $\D_2(B', B, S, T')$ in $\F$, a contradiction to the choice of $\F$.
 In the second case, we have diamond $\D_2(B, S, B\cup \{x\}, T')$. Because $B' \subsetneq B$ in this case, we have a contradiction to the choice of $B'$ as having maximum size. 
 
It remains to consider the case where $B = B'$. Under this assumption, we can conclude that $x \not \in S$ as $S \parallel (B\cup \{x\})$ and $B \subsetneq S$. Since $x \in T'\setminus S $, the pair $(S,T')$ in $\F$ contradicts the choice of $(B,T)$ because $B \subsetneq S$.

Having obtained a contradiction when $B$ and $B'$ were comparable and when they were incomparable, this concludes the proof of the claim.  
\end{proof}

We finish the proof of the lemma by showing that $\phi$ is injective. Suppose otherwise that there exists a pair $\{x,y\} \subseteq [n]$  and sets $A, A' \in \F$ with $\phi(x)=A \cup \{x\} = A' \cup \{y\}=\phi(y)$ in $\F$. The sets $A$ and $A'$ must be non-empty and  incomparable, since $x \in  A' \setminus A$ and $y \in  A \setminus A'$.  We conclude that $\D_2(\emptyset, A, A', A \cup \{x\})$ is a diamond in $\F$, which contradicts the hypothesis that $\F$ contains no induced copy of $\D_2$.  Hence $\phi:[n] \rightarrow  \F\setminus \{\emptyset\}$ is injective and  $|\F| \geq n+1$ as desired. 
	\end{proof}

We now establish the lower bound  for $\isat(n, \mathcal{V}_2)$ and $\isat(n, \Lambda_2)$ presented in Theorem~\ref{thm:veewedge}, which is a direct consequence of Lemma \ref{diamondempty}.
 
\begin{prop}\label{LB:veewedge}
	For  $n\geq 2$,
	\begin{align*}
		n+1\leq \isat(n, \mathcal{V}_2)= \isat(n, \Lambda_2).
	\end{align*}
\end{prop}

\begin{proof}
	We will  prove the result for $\Lambda_2$. The result for $\mathcal{V}_2$ follows by symmetry. Let $\F$ be an induced-$\Lambda_2$-saturated family. Then $\emptyset \in \F$  because $\emptyset$ is not contained in any induced copies of $\Lambda_2$ in $\B_n$. 

	It remains to show that $\F$ is also induced-$\D_2$-saturated.
	To see that $\F$ contains no  induced copy of a diamond, we remark that $\F$ contains no induced copy of $\Lambda_2$ that is an induced subposet of the diamond. 
	Hence, $\F$ contains no copy of the diamond. 
	For any $S \in  2^{[n]}\setminus \F$, there exists $A,B \in \F$ so that $S, A,$ and $B$  induce a $\Lambda_2$. 
	In particular, none of these three sets can be $\emptyset$. However, $S, A,$ and $B$, together with $\emptyset$ form an induced copy of $\D_2$ in $\F \cup S$. 
	Hence $\F$ is induced-$\D_2$-saturated. 
Thus $n+1 \leq \isat(n, \Lambda_2)$ is a direct consequence of Lemma~\ref{diamondempty}.
\end{proof}

\subsection{Separability}\label{sec:dist}

In this section, we introduce our main tool for establishing lower bounds on induced saturation numbers: the separability graph.  
Using this tool we will provide a lower bound on the induced saturation number for members of ${\mathbb{P}}$ (see Definition~\ref{phat}). 

\begin{definition}[Separating points]
	For any $n\in \mathbb{N}$ and for $x,y \in [n]$ we say that $S\subseteq {[n]}$ {\em separates} $x$ and $y$ if $|S\cap\{x,y\}|=1.$ 
\end{definition}
The study of set systems that separate every pair of points was initiated by R\'{e}nyi~\cite{R61}.

\begin{definition}[Separability Graph]\label{distgraph}
	Let $\F \subseteq 2^{[n]}$ be arbitrary.  
	Define $G_\F$ to be the graph with vertex set $[n]$ in which
	 distinct vertices $x$ and $y$ are adjacent if and only if $x$ and $y$ are separated by some $S \in \F$.
	We call $G_\F$ the  \textit{separability graph} of $\F$.
\end{definition}

Notice that if $\F\subseteq 2^{[n]}$ contains only one set $F$, then $G_{\F}$ is a complete bipartite graph with elements of $F$ in one partite class and elements of $ [n]\setminus F$ in the other partite class. 
Further, observe
\begin{align*}
	E(G_{\F}) = \bigcup\limits_{F \in \F} E(G_F).
\end{align*}

A \textit{biclique cover} of a graph $G$ is a collection of complete bipartite graphs $B_1,\dots,B_t$ such that $$\bigcup_{i=1}^t E(B_i) = E(G).$$  The minimum size of a biclique cover of a graph $G$ is the {\em biclique cover number}, denoted $\bc(G)$ (cf. \cite{MPNR1995, FH96}).  The bicliques $\{G_F :F\in\F\}$ form a biclique cover of $G_{\F}$, implying that

\begin{align*}
	\bc(G_{\F}) \leq |\{G_F:F \in \F\}| = |\F|.
\end{align*}

This establishes the following proposition and corollary.

\begin{prop}\label{distlb}
	Let $\F \subseteq 2^{[n]}$ be  arbitrary. If $G_{\F}$ is the separability graph of $\F$, then
	\begin{align*}
		|\F|\ge \bc(G_{\F}).
	\end{align*}
\end{prop}

The fact that $\bc(K_n) = \lceil \log_2 n \rceil$ yields the following corollary:
\begin{corollary}\label{dkn}
	Let $\F \subseteq 2^{[n]}$ be a family of sets such that for any pair $\{x,y\} \subseteq [n]$, there exists a set $F\in \F$ that separates $x$ and $y$.
	Then
	\begin{align*}
		|\F|\ge \lceil \log_2 n \rceil.
	\end{align*}
\end{corollary}
\begin{proof}
	If $\F$ separates every pair $\{x,y\} \subseteq [n]$, the graph $G_{\F}$ is isomorphic to $K_n$. 
	\end{proof}

Next, we prove that for each $\PP\in {\mathbb P}$, any induced-$\PP$-saturated family $\F$ separates every pair $\{x,y\} \subseteq [n]$, and so by Corollary~\ref{dkn} we have that $\lceil \log_2 n\rceil \leq \isat(n,\PP)$, establishing Theorem~\ref{phat}. 
To this end, we begin with the following lemma. 

\begin{lem}\label{logbnd}
	Let $\PP$ be in ${\mathbb P}$ and let $\F$ be an induced-$\PP$-saturated family in $\B_n$, where $n \geq 3$. 
	If there exists a pair $\{x,y\} \subseteq [n]$ so that no set $F \in \F$  separates $x$ and $y$, then $|F \cap\{x,y\}| = 0$ for all $F \in \F$.
\end{lem}

\begin{proof}
	 Fix a family $\F$ that is induced-$\PP$-saturated in $\B_n$, $n\geq 3$. Suppose there is a pair $\{x,y\} \subseteq [n]$ so that for every $F \in \F$,  $|F \cap \{x,y\}| \in \{0,2\}$. Define a partition $\F = \F_0 \cup \F_2$ where, for each $a \in \{0,2\}$, $\F_a$ contains those sets $F \in \F$ with $|F \cap \{x,y\}|=a$. 

	Towards a contradiction, suppose $\F_2$ is non-empty, and let $S$ be a  set of minimum size such that $S \cup \{x,y\} \in \F$. 
	 For ease of notation, let $S_{xy}$ denote $S \cup \{x,y\}$ and $S_x$ denote $S \cup \{x\}$. As $S_x \not \in \F$, there exists some $\F' \subseteq \F$ such that $\F' \cup \{S_x\}$ induces a copy of $\PP$. We will obtain a contradiction by showing that $\F' \cup \{S_{xy}\}$ induces a copy of $\PP$ in $\F$. First observe that for any $F\in\F'$, if $F \subset S_x$, then $F\subset S_{xy}$.  
	 
\begin{cl}\label{cl:UCTP1}	 
	For any $F \in \F'$, if $F \parallel S_x$, then $F\parallel S_{xy}$. 
\end{cl}
\begin{proof}
	 Indeed, if $F\in\F_0$, then there exists $i\in F\setminus S$, so that $F \parallel S_{xy}$. Otherwise, if $F\in F_2$, then $F=S' \cup \{x,y\}$ for some $S' \subseteq [n]\setminus \{x,y\}$.  By the choice of $S$, $|S'|\geq |S|$.  However, $S \not \subseteq S'$, as $F = S' \cup\{x,y\}$ and $S\cup \{x\}$ are incomparable by the hypothesis of this case.	Thus $S' \parallel S$ and so $F = S' \cup \{x,y\}$ and $S \cup \{x,y\}$ are incomparable as desired. 
\end{proof}
		
\begin{cl}\label{cl:UCTP2}
For any $F \in \F'$, if $F$ covers $S_x$ in the poset induced by $\F' \cup \{S_x$\}, then $F\neq S_{xy}$ and $F$ covers $S_{xy}$ in $\F' \cup \{S_{xy}\}$. 
\end{cl}
\begin{proof}
Let $A_1,\ldots,A_k\in\F'$ be the sets that cover $S_x$ in $\PP$, which are necessarily pairwise incomparable, and consequently form an antichain.  The fact that $|A_i\cap\{x,y\}|\neq 1$ and $S_x\subset A_i$ implies $S_{xy}\subseteq A_i$.  Further, if $k\geq 2$, then $S_{xy}$ must be a proper subset of each $A_i$, which implies the claim.
    
	Thus, suppose $A_1$ is the only set in $\F'$ that covers $S_x$. If $A_1\neq S_{xy}$, then $S_{xy} \subset A_1$ and $A_1$ covers $S_{xy}$ in $\F' \cup \{S_{xy}\}$. If instead $A_1=S_{xy}$, then recall $\PP$ has the UCTP. Thus there is a set $U\in\F'$ such that $A_1$ covers $U$ in the poset induced by $\F' \cup \{S_x\}$.
	If $U\in\F_0$, then $U\subset S$, which means $U\subset S_x$, contradicting the fact that $S_{xy}$ covers both $S_x$ and $U$.
	Thus, $U\in\F_2$, and so $U=S' \cup \{x,y\}$ for some $S' \subset S$. Since $U\in \F$ and $|S'| < |S|$, this contradicts the choice of $S$. It follows that $A_1 \neq S_{xy}$, as desired. 
\end{proof}
	
    Since $\F' \cup \{S_x\}$ induces a copy of $\PP$, the two claims imply that $\F' \cup \{S_{xy}\}$ also induces a copy of $\PP$. However, $S_{xy}\in \F$, which means that $\F$ contains an induced copy of $\PP$, a contradiction.
\end{proof}

The following lemma is a dual result to Lemma~\ref{logbnd}.

\begin{lem}\label{logbnd2}
	Let $\PP\in {\mathbb P}$ and let $\F$ be an induced-$\PP$-saturated family in $\B_n$, where $n \geq 3$. 
	If there exists a pair $\{x,y\} \subseteq [n]$ such that no set $F \in \F$  separates $x$ and $y$, then $\{x,y\}\subseteq F$ for all $F \in \F$.
\end{lem}

\begin{proof}
 Let $\PP\in {\mathbb P}$ and let $\F$ be an induced-$\PP$-saturated family in the Boolean lattice $\B_n$ such that there is a pair $\{x,y\}\subseteq [n]$ for which no set in $\F$ separates $x$ and $y$. Define $\PP^d$ to be the dual poset of $\PP$ where $S\leq T$ in $\PP^d$ if and only if $T\leq S$ in $\PP$. Define $\F^d$ to be the family $\{[n] \setminus  F : F\in \F\}$. 

First note that $\F^d$ is induced-$\PP^d$-saturated because $\F$ is induced-$\PP$-saturated. Further, $\F^d$ does not separate $x$ and $y$ because $\F$ does not. Finally, if an element $S$ in $\PP^d$ is covered by some $T$, then $S$ covers $T$ in $\PP$. 

Next we show that $\PP^d$ has the UCTP. Consider an element $S$ that is covered by $T$ in $\PP^d$. Since $\PP$ has the UCTP, either $T$ is covered by $S$ and $S'$ in $\PP$, or $T$ has exactly one cover $S$ in $\PP$ such that $S$ also covers $T'$ in $\PP$. The first case implies $S$ is covered by $T$ that also covers $S'$ in $\PP^d$. The second case implies $S$ is covered by both $T$ and $T'$ in $\PP^d$. Therefore $\PP^d$ has the UCTP. 

 By Lemma~\ref{logbnd}, for each $F^d$ in $\F^d$, $|F^d \cap \{x,y\}|=0$. Therefore $|F \cap \{x,y\}|=2$ for each $F\in \F$. 
\end{proof}

\begin{proof}[Proof of Theorem \ref{thm:phat}]
If $\PP$ is a non-trivial poset with the UCTP and $\F$ is an induced-$\PP$-saturated family that does not separate all pairs of points, then Lemmas ~\ref{logbnd} and ~\ref{logbnd2} imply that $\F=\emptyset$. However $\PP$ has at least two elements, so any induced-$\PP$-saturated family must have at least one element. Thus $\F$ must separate all pairs of points. By Corollary~\ref{dkn}, $|\F| \geq \lceil \log_2 n\rceil$, completing the proof.
\end{proof}

 Since each of $\D_2$, $\mathcal{N},$ and $\bowtie$ are in ${\mathbb{P}}$, Theorem~\ref{thm:phat} implies the lower bounds in Theorems~\ref{thm:diamond},~\ref{thm:theN}, and~\ref{thm:butterfly}.

\section{Future Work}
We have determined the exact induced-$\PP$-saturation number when $\PP$ is  $\mathcal{V}_2$ or $\Lambda_2$ and up to a constant when $\PP$ is $\mathcal{A}_{k+1}$. While many of our bounds do not match, we are prepared to conjecture the asymptotic value of $\isat(n,\PP)$ for several of the posets of interest in this paper.   

\begin{conj}\label{conj:ac}
        For $n>k\geq 3$, $\isat(n,\mathcal{A}_{k+1}) = kn  (1+o(1))$.
\end{conj}

\begin{conj}\label{conj:diamond}
	For $n \geq 2$, $\isat(n, \D_2) = \Theta(n).$
\end{conj}

\begin{conj}\label{conj:butterfly} 
	For $n\geq 3$, $\isat(n, \bowtie) = \Theta(n^2)$.
\end{conj}

Toward Conjecture~\ref{conj:diamond}, Lemma~\ref{diamondempty} implies that  we only require better lower bounds for induced-$\D_2$-saturated families that do not contain either $\emptyset$ or $[n]$. The family 
\begin{align*}
	\F=& \{1\} \cup \{\{1,i\}: i\in  [n]\setminus \{1\}\} 
	\cup \left( [n]\setminus \{1\}\right) \cup \{\left( [n]\setminus \{1,i\}\right): i\in  [n]\setminus \{1\}\}
\end{align*}
is such an induced-$\D_2$-saturated that avoids $\emptyset$ and $[n]$ and has size $2n$. Asymptotically we believe this to be the right answer. However, 
similar to the initial step (when $k=6$) of the iterative construction for Theorem~\ref{MNS}, 
it is likely that a small constant improvement can be made. In light of this, we pose the following conjecture: 

\begin{conj}\label{conj:2n}
	There exists a universal constant $c >0$ such that if  $\F$ is an induced-$\D_2$-saturated family that avoids $\emptyset$ and $[n]$, we have:
	\begin{align*}
		2n-c \leq |\F|.
	\end{align*}
\end{conj}

In this paper, we used the biclique cover number to provide a logarithmic lower bound for $\isat(n,\PP)$ whenever $\PP$ is non-trivial and has the UCTP (Theorem~\ref{thm:phat}). The biclique cover number result in Lemma~\ref{distlb} can be applied to any family $\F\subseteq 2^{[n]}$. It would be interesting to see Lemma~\ref{distlb} applied in a case where $G_{\F}\neq K_n$. 

More generally, there are many open questions in this subject. Here are a few to consider.

\begin{problem}
For which posets $\PP$ is $\sat(n,\PP)$ unbounded? 
\end{problem}

\begin{problem}
For which posets $\PP$ is $\isat(n,\PP)$ unbounded? 
\end{problem}

\begin{problem}
For which posets $\PP$ does the limit $\lim_{n\rightarrow \infty} \frac{\isat(n,\PP)}{n}$ exist?
\end{problem}

\begin{problem}
Is there function $f$ such that for any poset $\PP$, $\isat(n,\PP)\leq f(n)$?
\end{problem}

\section{Acknowledgements}
We would like to thank the referees for their helpful suggestions to improve the presentation of this paper. Their suggestions also led to improving the lower bounds in Theorem~\ref{thm:antichain} and Proposition~\ref{LB:antichain} to $3n-1$ from $2n$. One can improve this further to $4n-O(1)$, but the proof is much more tedious. Since this remains far from the upper bound, we leave further improvements for future work.

All authors were supported in part by NSF-DMS grant \#1427526, ``The Rocky Mountain-Great Plains Graduate Research Workshop in Combinatorics." Ferrara, Kay, Martin, and Smith were supported by a grant from the Simons Foundation (\#426971, Michael Ferrara). Martin was supported by a grant from the Simons Foundation (\#353292, Ryan R. Martin). Smith was supported in part by NSF-DMS grant \#1344199.

\section{References}
\bibliographystyle{abbrv}
\bibliography{Posat_Rev_DM}

\end{document}